\documentclass[12pt]{amsart}

\usepackage{amsmath}
\usepackage{amssymb}
\usepackage{ascmac}
\usepackage{amsthm}

\makeatletter

\@addtoreset{equation}{section}
\makeatother 

\theoremstyle{plain}
\newtheorem{thm}{Theorem}[section]
\newtheorem*{thm*}{Theorem}

\newtheorem{lem}[thm]{Lemma}
\newtheorem{cor}[thm]{Corollary}
\newtheorem{claim}[thm]{Claim}

\theoremstyle{remark}

\newcommand{\vol}{\operatorname{vol}}

\newcommand{\inrad}{\operatorname{inrad}}

\usepackage{latexsym}

\title[Some universal inequalities]{Some universal inequalities for Dirichlet eigenvalues of the Laplacian on a Euclidean convex domain}

\author{Kei Funano}
\address{Division of Mathematics \& Research Center for Pure and Applied Mathematics, Graduate School of Information Sciences, Tohoku University, 6-3-09 Aramaki-Aza-Aoba, Aoba-ku, Sendai 980-8579, Japan}
\email{kfunano@tohoku.ac.jp}

\date{\today}

\begin{document}
\maketitle

\begin{abstract}
We establish two universal inequalities for Dirichlet eigenvalues of the Laplacian on a Euclidean convex domain.
\end{abstract}






\section{Introduction}
Let $\Omega$ be a bounded domain in $\mathbb{R}^n$. 
We consider the Dirichlet eigenvalue problem for the Laplacian:
\begin{align*}
  \begin{cases}
    -\Delta u = \lambda u & \text{in } \Omega, \\
    u = 0 & \text{on } \partial\Omega.
  \end{cases}
\end{align*}
It is well known that the spectrum consists of an infinite sequence of eigenvalues $0 < \lambda_1(\Omega) < \lambda_2(\Omega) \leq \lambda_3(\Omega) \leq \dots \to \infty$. The asymptotic behavior of eigenvalues is described by the classical Weyl's law (\cite{W}, refer to \cite[Section 3.3]{LMP}):
\begin{align*}
    \lambda_k(\Omega)\sim 4\pi^2 \left(\frac{k}{\omega_n \vol(\Omega)}\right)^{\frac{2}{n}}
  \quad \text{as } k \to \infty,
\end{align*}where $\omega_n$ denotes the volume of the unit ball in $\mathbb{R}^n$. It is therefore a natural question whether the ratio $\lambda_k(\Omega)/\lambda_l(\Omega)$ 
can be bounded above and below by explicit multiples of $(k/l)^{2/n}$ with constants.

Such inequalities are referred to as \emph{universal inequalities}, 
as they hold for all domains in a given class without dependence on the particular geometry of $\Omega$. The study of universal inequalities for Dirichlet eigenvalues was pioneered by 
Payne, P\'{o}lya, and Weinberger \cite{PPW}. After that many mathematicians contributed and now the study of universal inequalities has a rich history (refer to \cite{A} and \cite[Section 5.4]{LMP}). A notable feature of these Dirichlet universal inequalities is that they hold 
for \emph{arbitrary} bounded domains, without any geometric restrictions. However, many of these universal inequalities focus on the ratio of consecutive eigenvalues or the sum of the first $k$ eigenvalues. It is more challenging to establish "ratios of arbitrary eigenvalues" that hold uniformly for a specific class of domains.

In this paper, we consider the Dirichlet eigenvalue problem on bounded convex domains 
in $\mathbb{R}^n$ and establish both upper and lower bounds 
for $\lambda_k(\Omega)/\lambda_l(\Omega)$ in terms of $(k/l)^{2/n}$.
Our main results are the following two theorems.
\begin{thm}\label{MTHM}There exists a constant $c_n>0$ depending only on $n$ which satisfies the following. Let $\Omega$ be a bounded convex domain in $\mathbb{R}^n$. Then for any $k\geq l$ we have
    \begin{align*}
      \lambda_k(\Omega)\leq c_n\Big(\frac{k}{l}\Big)^{\frac{2}{n}}\lambda_l(\Omega).
    \end{align*}The constant $c_n$ can be chosen as $c_n=12 n^3(j_{\frac{n}{2}-1,1})^2$, where $j_{\nu,1}$ is the first positive zero of the Bessel function $J_{\nu}$ of the first kind of order $\nu$.
\end{thm}

\begin{thm}\label{THM2}There exist two constants $\alpha_{n}$ and $\beta_n$ depending only on $n$ which satisfy the following. Let $\Omega$ be a bounded convex domain in $\mathbb{R}^n$ and $k\geq l$. If $\lambda_l(\Omega)>\alpha_n \lambda_1(\Omega)$ then we have  
    \begin{align*}
      \beta_n \Big(\frac{k}{l}\Big)^{\frac{2}{n}} \lambda_l(\Omega)\leq \lambda_k(\Omega).
    \end{align*}The constants $\alpha_n,\beta_n$ can be taken as 
    \begin{align*}
        \alpha_n=\frac{4^2}{\pi^2}(j_{\frac{n}{2}-1,1})^2n^3 \text{ and }\beta_n=\frac{1}{12n^3(j_{\frac{n}{2}-1,1})^2}.
    \end{align*}
\end{thm}

In Section \ref{final sect} we compare Theorem \ref{MTHM} with preceeding known inequalities. To the best of the author's knowledge Theorem \ref{THM2} is the first result to give a lower bound of Dirichlet eigenvalues in terms of lower eigenvalues. Also without any condition, like $\lambda_l(\Omega)>\alpha_n\lambda_1(\Omega)$, Theorem \ref{THM2} does not hold in general. Actually for each $k$ consider $\Omega_k:=[0,k]\times [0,1]$. Then we have $\lambda_1(\Omega_k)=\pi^2(1+\frac{1}{k^2})\sim \text{ const.}$ and $\lambda_k(\Omega_k)=2\pi^2$.

As a corollary of Theorem \ref{MTHM} we obtain the following upper bound for multiplicities of Dirichlet eigenvalues. For any natural number $k$ the {\emph multiplicity} of the $k$th eigenvalue $\lambda_k(\Omega)$ is denoted by $m_k(\Omega)$.  
\begin{cor}For any bounded convex domain $\Omega$ in $\mathbb{R}^n$ and for any natural number $k$, if $\lambda_k(\Omega)>\alpha_n \lambda_1(\Omega)$ then we have
    \begin{align*}
        m_k(\Omega)\leq 12^{\frac{n}{2}}n^{\frac{3}{2}n}(j_{\frac{n}{2}-1,1})^n k.
    \end{align*}Here $\alpha_n$ is the same constant appeared in Theorem \ref{THM2}. 
\end{cor}
\section{Preliminaries}

The following lemma by Hatcher is useful to treat convex domains. He showed that bounded convex domains can be approximated by an orthotope in some sense by utilizing the John theorem which asserts that bounded convex domains can be approximated by ellipsoids (\cite[Theorem III]{J}). 

\begin{lem}[{Hatcher, \cite[Proposition 2.3]{H}}]\label{Hacher lem} Let $\Omega$ be a bounded convex domain in $\mathbb{R}^n$. Choosing the coordinate axes appropriately if necessary we can find a sequence $a_1,a_2,\cdots, a_n$ of positive real numbers such that  
    \begin{align*}
       \Big[-\frac{a_1}{\sqrt{n}},\frac{a_1}{\sqrt{n}}\Big]\times \cdots \times \Big[-\frac{a_n}{\sqrt{n}},\frac{a_n}{\sqrt{n}}\Big]\subseteq  \Omega\subseteq [-a_1 n, a_1 n]\times \cdots [-a_n n, a_n n].
    \end{align*}
\end{lem}

The above lemma together with domain monotonicity for Dirichlet eigenvalues suggest us that the proof of the main theorem reduces to the case where the domain is an orthotope. Domain monotonicity for Dirichlet eigenvalues asserts that larger domains have smaller eigenvalues (\cite[Theorem 3.2.1]{LMP}):
\begin{align*}
    \Omega_1\subseteq \Omega_2 \Rightarrow \lambda_k(\Omega_2)\leq \lambda_k(\Omega_1) \text{ for any }k.
\end{align*}

For a bounded domain $\Omega$ recall that the \emph{inradius} $\rho_{\Omega}$ of $\Omega$ is the radius of the largest ball contained inside $\Omega$. From domain monotonicity it is known the following upper bound for the $1$st Dirichlet eigenvalue in terms of the inradius. 
\begin{align}
    \lambda_1(\Omega)\leq \frac{(j_{\frac{n}{2}-1,1})^2}{\inrad (\Omega)^2}.
\end{align}

When the domain is convex, Hersch \cite{He} and Protter \cite{Pr} proved the following reverse inequality:
\begin{thm}[{Hersch-Protter, \cite[Proposition 2.55]{FLW}}]\label{HPr est}For a bounded convex domain $\Omega$ in $\mathbb{R}^n$ we have
    \begin{align*}
        \lambda_1(\Omega)\geq \frac{\pi^2}{4\inrad(\Omega)^2}.
    \end{align*}
\end{thm}

We also use the following lower bound for Dirichlet eigenvalues in terms of volume due to Berezin \cite{B} and Li-Yau \cite{LY}.   
\begin{thm}[{Berezin-Li-Yau, \cite[Corollary 3.3.22]{LMP}}]\label{BLY ineq}Let $\Omega$ be a (not necessarily convex) bounded domain in $\mathbb{R}^n$. Then for any natural number $k$ we have
\begin{align*}
    \lambda_k(\Omega)\geq \frac{(2\pi)^2n}{n+2}\Big(\frac{k}{\omega_n \vol \Omega}\Big)^{\frac{2}{n}}.
\end{align*}
\end{thm}

\section{Proof of the Main Theorems}
In this section we prove the main theorems. 
The following lemma is a key to prove Theorem \ref{MTHM}. 
\begin{lem}\label{Mlem1}There exists a constant $c_n>0$ depending only on $n$ which satisfies the following. For any orthotope $R$ in $\mathbb{R}^n$ and any $k\geq l$ we have 
    \begin{align*}
           \lambda_k(R)\leq c_n \Big(\frac{k}{l}\Big)^{\frac{2}{n}}\lambda_l(R).
    \end{align*}The constant $c_n$ can be chosen as $c_n= 12(j_{\frac{n}{2}-1,1})^2$.
\end{lem}

\begin{proof}[Proof of Theorem \ref{MTHM}]By Hatcher's lemma (Lemma \ref{Hacher lem}), after choosing appropriate coordinate axes, we can find an orthotope $R$ such that
$\frac{1}{\sqrt{n}}R\subseteq \Omega \subseteq nR$. Let the constant $c_n:=12 (j_{\frac{n}{2}-1,1})^2$. Applying Lemma \ref{Mlem1} and using domain monotonicity for Dirichlet eigenvalues we obtain
\begin{align*}
    \lambda_k(\Omega)\leq \lambda_k\Big(\frac{1}{\sqrt{n}}R\Big)=n\lambda_k(R)\leq c_n n\Big(\frac{k}{l}\Big)^{\frac{2}{n}}\lambda_l(R)=\ &c_n n^3 \Big(\frac{k}{l}\Big)^{\frac{2}{n}}\lambda_l(nR)\\ \leq \ & c_n n^3 \Big(\frac{k}{l}\Big)^{\frac{2}{n}}\lambda_l(\Omega).
\end{align*}This completes the proof of the theorem.
\end{proof}

\begin{proof}[Proof of Lemma \ref{Mlem1}]
   Put $r:=\frac{2}{\sqrt{3}} \big(\frac{l}{k}\big)^{\frac{1}{n}}\frac{1}{\sqrt{\lambda_l(R)}}$. Changing the coordinate axes if neccessary we may assume that $R=[-a_1,a_1]\times [-a_2,a_2]\times \cdots \times [-a_n,a_n]$ and $a_1\leq a_2\leq \cdots \leq a_n$.

Take a maximal $r$-separated points $\{x_i\}_{i=1}^{k'}$ in $R$. The maximality gives $R\subseteq \bigcup_{i=1}^{k'}B(x_i,r)$ which shows
\begin{align*}
    \vol R \leq \sum_{i=1}^{k'}\vol B(x_i,r)=k'\omega_nr^n= \Big(\frac{2}{\sqrt{3}}\Big)^n\cdot \frac{\omega_n k'l}{k\lambda_l(R)^{\frac{n}{2}}}.
\end{align*}By virtue of the Berezin-Li-Yau inequality (Theorem \ref{BLY ineq}) we have
\begin{align*}
    \vol R \leq \Big(\frac{2}{\sqrt{3}}\Big)^n\cdot \Big(\frac{2+n}{(2\pi)^2n}\Big)^{\frac{n}{2}}\frac{\omega_n^2k'}{k}\vol R,
\end{align*}which leads to
\begin{align*}
    k\leq \Big(\frac{2}{\sqrt{3}}\Big)^n\cdot \Big(\frac{2+n}{(2\pi)^2n}\Big)^{\frac{n}{2}}\omega_n^2k'.
\end{align*}
Since 
\begin{align*}
    \Big(\frac{(2\pi)^2n}{2+n}\Big)^{\frac{1}{2}}\omega_n^{-\frac{2}{n}}=2\Big(\frac{n}{2+n}\Big)^{\frac{1}{2}}\Gamma\Big(\frac{n}{2}+1\Big)^{\frac{2}{n}}\geq 2\Big(\frac{n}{2+n}\Big)^{\frac{1}{2}}\geq \frac{2}{\sqrt{3}},
\end{align*}we have $k\leq k'$.

\begin{claim}$a_1\geq r/2$
\end{claim}

Let us admit the claim for a while. Put $A_i=B(x_i,\frac{r}{2})\cap R$ for $i=1,\cdots,k'$. Then since $\{x_i\}_{i=1}^{k'}$ is $r$-sparated these balls are disjoint each other. Thanks to the claim the inradius of $A_i$ is at least $\frac{r}{4}$ (Here we use that our domain $R$ is an orthotope). As a consequence we obtain
\begin{align*}
    \lambda_k(R)\leq \lambda_{k'}(R)\leq \max_{i}\lambda_1(A_i)\leq  \frac{(j_{\frac{n}{2}-1,1})^2}{(\frac{r}{4})^2}=12(j_{\frac{n}{2}-1,1})^2\Big(\frac{k}{l}\Big)^{\frac{2}{n}}\lambda_l(R),
\end{align*}which concludes the lemma.

We now prove the claim. Note that $a_1=\inrad R$. Thus the Hersch-Protter inequality (Theorem \ref{HPr est}) gives
\begin{align*}
    a_1\geq \frac{\pi}{2\sqrt{\lambda_1(R)}}\geq \frac{\pi l^{\frac{1}{n}}}{2k^{\frac{1}{n}}\sqrt{\lambda_l(\Omega)}}>\frac{r}{2},
\end{align*}which implies the claim. This completes the proof of the lemma.
\end{proof}

\begin{lem}\label{Mlem2}
    There exist two constants $c_{n}$ and $d_n$ depending only on $n$ which satisfy the following. Let $R$ be an orthotope in $\mathbb{R}^n$ and $k\geq l$. If $\lambda_l(R)>c_n \lambda_1(R)$ then we have  
    \begin{align*}
      d_n \Big(\frac{k}{l}\Big)^{\frac{2}{n}} \lambda_l(R)\leq \lambda_k(R).
    \end{align*}The constants $c_n,d_n$ can be taken as 
    \begin{align}\label{const}
        c_n=\frac{4^2}{\pi^2}(j_{\frac{n}{2}-1,1})^2 \text{ and }d_n=\frac{1}{12(j_{\frac{n}{2}-1,1})^2}.
    \end{align}
\end{lem}
As the proof of Theorem \ref{MTHM}, Theorem \ref{THM2} follows from Lemma \ref{Mlem2} together with
domain monotonicity and Hatcher's lemma (Lemma \ref{Hacher lem}), so we omit the proof.

\begin{proof}[Proof of Lemma \ref{Mlem2}] Suppose that $\lambda_l(R)>c_n\lambda_1(R)$ and put $r:=\frac{2}{\sqrt{3}} \big(\frac{k}{l}\big)^{\frac{1}{n}}\frac{1}{\sqrt{\lambda_k(R)}}$, where we take $c_n$ as in (\ref{const}). We may assume that $R=[-a_1,a_1]\times [-a_2,a_2]\times \cdots \times [-a_n,a_n]$ and $a_1\leq a_2\leq \cdots \leq a_n$.

Take a maximal $r$-separated points $\{x_i\}_{i=1}^{l'}$ in $R$. Then as in the proof of Lemma \ref{Mlem1} we have $l\leq l'$. 

Suppose that $a_1\leq \frac{r}{2}$. Putting $A_i:=B(x_i,a_1)\cap R$ we see that $\{A_i\}_{i=1}^{l'}$ is a family of disjoint subsets in $R$. Note that $\inrad A_i\geq \frac{a_1}{2}$. We thereby get
\begin{align*}
    \lambda_l(R)\leq \lambda_{l'}(R)\leq \max_i \lambda_1(A_i)\leq \frac{(j_{\frac{n}{2}-1,1})^2}{(\frac{a_1}{2})^2}.
\end{align*}Using the Hersch-Protter inequality (Theorem \ref{HPr est}) we have
\begin{align*}
    \lambda_l(R)\leq \frac{4^2}{\pi^2}(j_{\frac{n}{2}-1,1})^2 \lambda_1(R)=c_n\lambda_1(R),
\end{align*}which contradicts to our assumption $\lambda_l(R)>c_n\lambda_1(R)$. 

Therefore we obtain $a_1>\frac{r}{2}$. Putting $B_i=B(x_i,\frac{r}{2})\cap R$ we thus find that $\{B_i\}_{i=1}^{k'}$ is a family of disjoint subsets in $R$ such that $\inrad B_i\geq \frac{r}{4}$. As a result we obtain
\begin{align*}
    \lambda_l(R)\leq \lambda_{l'}(R)\leq \max_i\lambda_1(B_i)\leq \frac{(j_{\frac{n}{2}-1,1})^2}{(\frac{r}{4})^2} = 12(j_{\frac{n}{2}-1,1})^2\Big(\frac{l}{k}\Big)^{\frac{2}{n}}\lambda_k(R).
\end{align*}This completes the proof.
\end{proof}

\section{Previous results}\label{final sect}
In this section we overview preceeding results related to main theorems.

For any bounded domain $\Omega$ in $\mathbb{R}^n$, Payne-P\'{o}lya-Weinberger \cite{PPW} (see also \cite[Theorem 5.4.1]{LMP}) proved the following universal inequality:
\begin{align}\label{uni1}
    \lambda_{k+1}(\Omega)-\lambda_k(\Omega)\leq \frac{1}{nk}\sum_{i=1}^k \lambda_i(\Omega) \text{ for any }k,
\end{align}which leads to 
\begin{align}\label{uni2}
    \lambda_{k+1}(\Omega)\leq \Big(1+\frac{1}{n}\Big)\lambda_k(\Omega) \text{ for any }k.
\end{align}
Later Hille-Protter \cite{HP} improved this inequality and then Yang \cite{Y} proved an even stronger inequality. From Yang's inequality one can derive that
\begin{align}\label{uni3}
    \lambda_{k+1}(\Omega)\leq \Big(1+\frac{4}{n}\Big)\frac{1}{k}\sum_{i=1}^k \lambda_i(\Omega) \text{ for any }k,
\end{align}which is known as Yang's second inequality (\cite[5.4.4]{LMP}). In particular this inequality derives that 
\begin{align*}
    \lambda_{k+1}(\Omega)\leq \Big(1+\frac{4}{n}\Big)\lambda_k(\Omega) \text{ for any }k.
\end{align*}Note that this inequality is weaker than (\ref{uni1}) although (\ref{uni3}) is not the consequence of (\ref{uni1}).

For any bounded domain $\Omega$ in $\mathbb{R}^n$ Levitin-Parnovski \cite[(4.14)]{LP} proved that 
\begin{align*}
    \sum_{i=1}^n\lambda_{i+k}(\Omega)\leq 
   (4+n) \lambda_k(\Omega) \text{ for any }k.
\end{align*}
This yields that
\begin{align}\label{uni4}
    \lambda_{i+k}(\Omega)\leq \frac{4+n}{n-i+1}\lambda_k(\Omega) \text{ for any }k \text{ and for any }i\leq n.
\end{align}
Ashbaugh-Benguria \cite{Ab} resolved the conjecture by Payne-P\'olya-Weinberger concerning the ratio of the first and the second eigenvalue:
\begin{align}\label{uni5}
    \frac{\lambda_2(\Omega)}{\lambda_1(\Omega)}\leq \frac{\lambda_2(B(0,1))}{\lambda_1(B(0,1))}=\frac{j_{\frac{n}{2},1}^2}{j_{\frac{n}{2}-1,1}^2}.
\end{align}They also showed that the equality in (\ref{uni5}) occurs if and only if $\Omega$ is a disk.

Cheng-Yang (\cite[Theorem 3.1]{CY}) proved that for any bounded domain $\Omega$ in $\mathbb{R}^n$ and for any $n\geq 41$ and $k\geq 41$, 
\begin{align}\label{uni6}
    \lambda_{k+1}(\Omega)\leq k^{\frac{2}{n}}\lambda_1(\Omega).
\end{align}They also proved that for any general $n$ and $k$, 
\begin{align*}
    \lambda_{k+1}(\Omega)\leq C_0(n,k)k^{\frac{2}{n}}\lambda_1(\Omega),
\end{align*}where
\begin{align*}
    C_{0}(n, k) = \begin{cases}
\dfrac{j_{n/2, 1}^{2}}{j_{n/2-1, 1}^{2}}, & \text{for } k = 1 \\
1 + \dfrac{a(\min\{n, k - 1\})}{n}, & \text{for } k \geq 2
\end{cases}
\end{align*}
and $a(1) \leq 2.64$, $a(2) \leq  2.27$ and $a(p)\leq 2.2-4\log (1+\frac{p-3}{50})$ for $p\geq 3$ is a
constant depending only on $p$. For $k\geq 2$ this yields
\begin{align}\label{uni7}
    \lambda_{k+1}(\Omega)\leq \ & \Big(1+\max\Big\{\frac{2.64}{n},2.24\log \Big(1+\frac{n-3}{50}\Big)\Big\}\Big)k^{\frac{2}{n}}\lambda_1(\Omega)\\ \leq \ &\Big(1+\max\Big\{\frac{2.64}{n},\frac{2.24}{50}\cdot\frac{n-3}{n}\Big\} \Big) k^{\frac{2}{n}}\lambda_1(\Omega). \notag{}
\end{align}
Note that none of the above results make any assumptions about the domain wheares in our theorems (Theorems \ref{MTHM} and \ref{THM2}) we assume that our domain is convex. 
Since 
\begin{align*}
    \sqrt{(\nu+1)(\nu+5)}<j_{\nu,1}<\sqrt{2(\nu+1)(\nu+3)}
\end{align*}(\cite[(1)]{Le} and \cite[p.486 (5)]{Wa}), we see that Theorem \ref{MTHM} is weaker than (\ref{uni2}--\ref{uni7}) in each case. The advantage of Theorem \ref{MTHM} is that it unifies these inequalities up to a constant factor under the convexity assumption. At present, the author does not know whether the assumption of convexity is necessary for Theorems \ref{MTHM} and \ref{THM2}.

\emph{Acknowledgement.}This work was supported by JSPS KAKENHI Grant Number JP24K06731.


\end{document}